\documentclass[english]{article}
\usepackage[T1]{fontenc}
\usepackage[latin9]{inputenc}
\usepackage{mathrsfs}
\usepackage{amsmath}
\usepackage{amssymb}

\makeatletter

\providecommand{\tabularnewline}{\\}

\usepackage[T1]{fontenc}
\usepackage[latin9]{inputenc}
\usepackage{color}
\usepackage{amsmath}
\usepackage{graphicx}
\usepackage{amsfonts}
\usepackage{a4wide}
\usepackage{babel}%

\usepackage{mathrsfs}
\usepackage[active]{srcltx}

\providecommand{\U}[1]{\protect\rule{.1in}{.1in}}

\newtheorem{theorem}{Theorem}

\newtheorem{corollary}[theorem]{Corollary}

\newtheorem{definition}[theorem]{Definition}

\newtheorem{lemma}[theorem]{Lemma}

\newtheorem{proposition}[theorem]{Proposition}
\newtheorem{remark}[theorem]{Remark}

\newenvironment{proof}[1][Proof]{\noindent\textbf{#1.} }{\ \rule{0.5em}{0.5em}}

\usepackage[usenames,dvipsnames]{pstricks}
\usepackage{epsfig}
\usepackage{pst-grad} 
\usepackage{pst-plot} 

\usepackage{graphicx}

\makeatother

\usepackage{babel}
\begin{document}

\title{Location, identification, and representability of monotone operators
in locally convex spaces}

\author{M.D. Voisei}

\date{{}}
\maketitle
\begin{abstract}
In this paper we study, in the relaxed context of locally convex spaces,
intrinsic properties of monotone operators needed for the sum conjecture
for maximal monotone operators to hold under classical interiority-type
domain constraints.
\end{abstract}

\section{Introduction and preliminaries}

The aim of this note is to reveal deeper properties of maximal montone
operators that enjoy, under a locally convex space settings, the classical
sum theorem which, in the literature, is sometimes called, when the
context is provided by Banach spaces, the Rockafellar conjecture. 

A breakthrough in the study of maximal monotone operators is represented
by the introduction, in 2006 in \cite[Theorem\ 2.3]{MR2207807}, of
a new characterization of maximal monotonicity based on the notions
of representability and ``NI-type'' operator (see \cite[Remark 3.5]{MR2453098}
or Remark \ref{NI method} below for more details). This characterization
works in locally convex spaces and, is the main argument used after
2006 in the majority of the articles concerning the calculus rules
for maximal monotone operators in general Banach spaces such as those
in \cite{MR2583910,MR2207807,MR2230783,MR2389004,MR2453098,astfnc,mmicq,MR2577332}. 

The present paper enhances the aforementioned maximality characterization
by presenting localized versions of it together with their direct
consequences.

The plan of the paper is as follows. Section $2$ presents the three
main notions studied in this article together with their immediate
properties and some variants. Section $3$ is concerned with the interplay
of these notions. Section $4$ contains the representability of the
sum of two representable operators. We conclude our article with some
open problems in Section $5$. 

\medskip

Throughout this paper, if not otherwise explicitly mentioned, $(X,\tau)$
is a non-trivial (that is, $X\neq\{0\}$) Hausdorff separated locally
convex space (LCS for short), $X^{\ast}$ is its topological dual
endowed with the weak-star topology $\omega^{\ast}$, the topological
dual of $(X^{\ast},\omega^{\ast})$ is identified with $X$, and the
weak topology on $X$ is denoted by $\omega$. 

We denote by $\mathscr{V}_{\tau}(x)$ the family of $\tau-$neighborhoods
of $x\in X$ and the convergence of nets in $(X,\tau)$ by $x_{i}\stackrel{\tau}{\to}x$. 

The \emph{duality product} or \emph{coupling} of $X\times X^{\ast}$
is denoted by $\left\langle x,x^{\ast}\right\rangle :=x^{\ast}(x)=:c(x,x^{\ast})$,
for $x\in X$, $x^{\ast}\in X^{\ast}$. As usual, with respect to
the dual system $(X,X^{*})$, we denote the \emph{orthogonal of} $S\subset X$
by $S^{\perp}:=\{x^{*}\in X^{*}\mid\langle x,x^{*}\rangle=0,$ for
every $x\in S\}$ and the \emph{support function} of $S$ by $\sigma_{S}(x^{*}):=\sup_{x\in S}\langle x,x^{*}\rangle$,
$x^{*}\in X^{*}$ while for $M\subset X^{*}$, the orthogonal of $M$
is denoted by $M^{\perp}:=\{x\in X\mid\langle x,x^{*}\rangle=0,$
for every $x\in M\}$ and its support function is $\sigma_{M}(x)=\sup_{x^{*}\in M}\langle x,x^{*}\rangle$,
$x\in X$.

To a multi-valued operator $T:X\rightrightarrows X^{\ast}$ we associate
its 
\begin{itemize}
\item \emph{graph}: $\operatorname*{Graph}T=\{(x,x^{\ast})\in X\times X^{\ast}\mid x^{\ast}\in Tx\}$, 
\item \emph{inverse:} $T^{-1}:X^{*}\rightrightarrows X$, $\operatorname*{gph}T^{-1}=\{(x^{*},x)\mid(x,x^{\ast})\in\operatorname*{Graph}T\}$, 
\item \emph{domain}: $D(T):=\{x\in X\mid Tx\neq\emptyset\}=\Pr\nolimits _{X}(\operatorname*{Graph}T)$,
and 
\item \emph{range}: $R(T):=\{x^{*}\in X^{*}\mid x^{*}\in Tx\ {\rm for\ some}\ x\in X\}=\Pr\nolimits _{X^{*}}(\operatorname*{Graph}T)$.
\\ Here $\operatorname*{Pr}{}_{X}$ and $\operatorname*{Pr}{}_{X^{*}}$
are the projections of $X\times X^{*}$ onto $X$ and $X^{\ast}$,
respectively. 
\item \emph{direct image:} $T(A):=\cup_{x\in A}Tx$, $A\subset X$. 
\end{itemize}
When no confusion can occur, $T:X\rightrightarrows X^{\ast}$ will
be identified with $\operatorname*{Graph}T\subset X\times X^{*}$.

\medskip

In the sequel, given a locally convex space $(E,\tau)$ and $S\subset E$,
the following notations are used: ``$\operatorname*{cl}_{\tau}S=\overline{S}^{\tau}$''
for the $\tau-$\emph{closure} of $S$, ``$\operatorname*{int}_{\tau}S$''
for the $\tau-$topological interior of $S$, ``$\operatorname*{bd}_{\tau}S=\overline{S}^{\tau}\setminus\operatorname*{int}_{\tau}S$''
for the boundary of $S$, ``$\operatorname*{conv}S$'' for the \emph{convex
hull} of $S$, ``$\operatorname*{aff}S$'' for the \emph{affine
hull} of $S$, and ``$S^{i}=\operatorname*{core}S$'' for the \emph{algebraic
interior} of $S$, ``$^{i}S$'' for the \emph{relative algebraic
interior} of $S$ with respect to $\operatorname*{aff}S$. When the
topology $\tau$ is implicitly understood the use of the $\tau-$notation
is avoided. 

A set $S\subset E$ is called \emph{algebraically open} if $S=\operatorname*{core}S$. 

We denote by $\iota_{S}$ the \emph{indicator} \emph{function} of
$S\subset E$ defined by $\iota_{S}(x):=0$ for $x\in S$ and $\iota_{S}(x):=\infty$
for $x\in E\setminus S$. 

The set $[x,y]:=\{tx+(1-t)y\mid0\le t\le1\}\subset E$ represents
the closed segment with end-points $x,y\in E$. 

For $f,g:E\rightarrow\overline{\mathbb{R}}:=\mathbb{R}\cup\{-\infty,+\infty\}$
we set $[f\leq g]:=\{x\in E\mid f(x)\leq g(x)\}$, $[f=g]$, $[f<g]$
and $[f>g]$ are similarly defined, while e.g. $f\ge g$ means $[f\ge g]=E$
or, for every $e\in E$, $f(e)\ge g(e)$. 

We consider the following classes of functions and operators on $X$: 
\begin{description}
\item [{$\Lambda(X)$}] is the class of proper convex functions $f:X\rightarrow\overline{\mathbb{R}}$.
Recall that $f$ is \emph{proper} if $\operatorname*{dom}f:=\{x\in X\mid f(x)<\infty\}$
is nonempty and $f$ does not take the value $-\infty$; 
\item [{$\Gamma_{\tau}(X)$}] is the class of functions $f\in\Lambda(X)$
that are $\tau$\textendash \emph{lower semi}\-\emph{continuous}
(\emph{$\tau$\textendash }lsc for short); when the topology is implicitly
understood the notation $\Gamma(X)$ is used instead; 
\item [{$\mathcal{M}(X)$}] is the class of non-void monotone operators
$T:X\rightrightarrows X^{\ast}$ ($\operatorname*{Graph}T\neq\emptyset$).
Recall that $T:X\rightrightarrows X^{\ast}$ is \emph{monotone} if
$\left\langle x_{1}-x_{2},x_{1}^{\ast}-x_{2}^{\ast}\right\rangle \geq0$
for all $x_{1},x_{2}\in D(T)$, $x_{1}^{\ast}\in Tx_{1}$, $x_{2}^{\ast}\in Tx_{2}$; 
\item [{$\mathfrak{M}(X)$}] is the class of \emph{maximal monotone} operators
$T:X\rightrightarrows X^{\ast}$. The maximality is understood in
the sense of graph inclusion as subsets of $X\times X^{\ast}$. 
\end{description}
Recall some notions associated to a proper function $f:X\rightarrow\overline{\mathbb{R}}$: 
\begin{description}
\item [{$\operatorname*{epi}f:=\{(x,t)\in X\times\mathbb{R}\mid f(x)\leq t\}$}] is
the \emph{epigraph} of $f$; 
\item [{$\operatorname*{cl}_{\tau}f:X\rightarrow\overline{\mathbb{R}},$}] the
$\tau$\emph{\textendash lsc hull} of $f$, is the greatest $\tau$\textendash lsc
function majorized by $f$; $\operatorname*{epi}(\operatorname*{cl}_{\tau}f)=\operatorname*{cl}_{\tau}(\operatorname*{epi}f);$ 
\item [{$\operatorname*{conv}f:X\rightarrow\overline{\mathbb{R}}$,}] the
\emph{convex hull} of $f$, is the greatest convex function majorized
by $f$; $(\operatorname*{conv}f)(x):=\inf\{t\in\mathbb{R}\mid(x,t)\in\operatorname*{conv}(\operatorname*{epi}f)\}$
for $x\in X$; 
\item [{$\overline{\operatorname*{conv}}^{\tau}f:X\rightarrow\overline{\mathbb{R}}$,}] the
\emph{$\tau-$lsc convex hull} of $f$, is the greatest \emph{$\tau$\textendash }lsc
convex function majorized by $f$; $\operatorname*{epi}(\overline{\operatorname*{conv}}^{\tau}\!f):=\overline{\operatorname*{conv}}^{\tau}(\operatorname*{epi}f)$; 
\item [{$f^{\ast}:X^{\ast}\rightarrow\overline{\mathbb{R}}$}] is the \emph{convex
conjugate} of $f:X\rightarrow\overline{\mathbb{R}}$ with respect
to the dual system $(X,X^{\ast})$, $f^{\ast}(x^{\ast}):=\sup\{\left\langle x,x^{\ast}\right\rangle -f(x)\mid x\in X\}$
for $x^{\ast}\in X^{\ast}$; 
\item [{$\partial f(x)$}] is the \emph{subdifferential} of the proper
function $f:X\rightarrow\overline{\mathbb{R}}$ at $x\in X$; $\partial f(x):=\{x^{\ast}\in X^{\ast}\mid\left\langle x^{\prime}-x,x^{\ast}\right\rangle +f(x)\leq f(x^{\prime}),\ \forall x^{\prime}\in X\}$
for $x\in X$ (it follows from its definition that $\partial f(x):=\emptyset$
for $x\not\in\operatorname*{dom}f$). Recall that $N_{C}=\partial\iota_{C}$
is the \emph{normal cone} to $C$, where $\iota_{C}(x)=0$, if $x\in C$,
$\iota_{C}(x)=+\infty$ otherwise; $\iota_{C}$ is the indicator\emph{
}function of $C\subset X$. 
\end{description}
For $(X,\tau)$ a LCS, let $Z:=X\times X^{\ast}$. It is known that
$(Z,\tau\times\omega^{\ast})^{\ast}=Z$ via the coupling 
\[
z\cdot z^{\prime}:=\left\langle x,x^{\prime\ast}\right\rangle +\left\langle x^{\prime},x^{\ast}\right\rangle ,\quad\text{for }z=(x,x^{\ast}),\ z^{\prime}=(x^{\prime},x^{\prime\ast})\in Z;
\]
$(Z,Z)$ is called the \emph{natural dual system}.

For a proper function $f:Z\rightarrow\overline{\mathbb{R}}$ all the
above notions are defined similarly. In addition, with respect to
the natural dual system $(Z,Z)$, the conjugate of $f$ is given by
\[
f^{\square}:Z\rightarrow\overline{\mathbb{R}},\quad f^{\square}(z)=\sup\{z\cdot z^{\prime}-f(z^{\prime})\mid z^{\prime}\in Z\},
\]
 and by the biconjugate formula, $f^{\square\square}=\overline{\operatorname*{conv}}^{\tau\times\omega^{\ast}}\!\!f$
whenever $f^{\square}$ (or $\overline{\operatorname*{conv}}^{\tau\times\omega^{\ast}}\!\!f$)
is proper.

We introduce the following classes of functions:
\begin{align*}
\mathscr{C}: & =\mathscr{C}(Z):=\{f\in\Lambda(Z)\mid f\geq c\},\\
\mathscr{R}: & =\mathscr{R}(Z):=\Gamma_{\tau\times\omega^{\ast}}(Z)\cap\mathscr{C}(Z),\\
\mathscr{D}: & =\mathscr{D}(Z):=\{f\in\mathscr{R}(Z)\mid f^{\square}\geq c\}.
\end{align*}

It is known that $[f=c]\in\mathcal{M}(X)$ for every $f\in\mathscr{C}(Z)$
(see e.g. \cite[Proposition 4(h)]{MR2086060}, \cite[Lemma\ 3.1]{MR2453098}).

\begin{lemma} \label{h>c} Let $X$ be a LCS and let $h\in\mathscr{C}$.
Then $[h=c]\subset[h^{\square}=c]$. If, in addition, $h\in\mathscr{D}$
then $[h=c]=[h^{\square}=c]$. \end{lemma}

\begin{proof} Let $z\in[h=c]$. Then
\[
h'(z;w)=\lim_{t\downarrow0}\frac{h(z+tw)-h(z)}{t}\ge\lim_{t\downarrow0}\frac{c(z+tw)-c(z)}{t}=z\cdot w,\ \forall w\in Z,
\]
which shows that $z\in\partial h(z)$. Therefore $h(z)+h^{\square}(z)=z\cdot z=2c(z)$
and so $h^{\square}(z)=c(z)$. 

If, in addition, $h\in\mathscr{D}$ the stated equality follows from
$h=h^{\square\square}$ and the previously shown inclusion applied
for $h$ and $h^{\square}$. \end{proof}

\strut

To a multi\-function $T:X\rightrightarrows X^{\ast}$ we associate
the following functions: $c_{T}:Z\rightarrow\overline{\mathbb{R}}$,
$c_{T}:=c+\iota_{\operatorname*{Graph}T}$, $\psi_{T}:Z\rightarrow\overline{\mathbb{R}}$,
$\psi_{T}:=\operatorname*{cl}_{\tau\times\omega^{\ast}}(\operatorname*{conv}c_{T})$,
$\varphi_{T}:Z\rightarrow\overline{\mathbb{R}}$, $\varphi_{T}:=c_{T}^{\square}=\psi_{T}^{\square}\ -$
the \emph{Fitzpatrick function} of $T$. In expanded form 
\[
\varphi_{T}(z):=\varphi_{T}(x,x^{*}):=\sup\{z\cdot w-c(w)\mid w\in T\}=\sup\{\langle x,u^{*}\rangle+\langle u,x^{*}\rangle-\langle u,u^{*}\rangle\mid(u,u^{*})\in T\},
\]
 for $z=(x,x^{*})\in Z$. The function $\varphi_{T}$, $\psi_{T}$
were introduced first in \cite{MR1009594}, \cite{MR2207807}.

Recall that whenever $T\in\mathcal{M}(X)$, $\varphi_{T},\psi_{T}\in\Gamma_{\tau\times w^{*}}(Z)$. 

The set $T^{+}:=[\varphi_{T}\le c]$ describes all elements of $Z$
that are \emph{monotonically related} (m.r. for short) \emph{to} $T$.

Let us recall several properties of these functions.

\begin{theorem}\label{propr} Let $X$ be a LCS. \\\vspace{-.2cm}

\noindent \emph{(i)} For every $T\subset X\times X^{*}$, $T\subset(D(T)\times X^{*})\cup(X\times R(T))\subset[\varphi_{T}\ge c]$
\emph{(see \cite[Theorem\ 1.1]{MR2207807}, \cite[Proposition\ 3.2]{tscr-arxiv,MR2389004})},

\medskip

\noindent \emph{(ii)} $T\in\mathcal{M}(X)$ iff $T\subset[\varphi_{T}=c]$
iff $\psi_{T}\ge c$, \emph{(\cite[Proposition\ 3.2]{tscr-arxiv,MR2389004},
\cite[(8)]{MR2577332})}

\medskip

\noindent \emph{(iii)} $T\in\mathfrak{M}(X)$ iff $T\in\mathcal{M}(X)$,
$T=[\psi_{T}=c]$, and $\varphi_{T}\ge c$ \emph{(\cite[Theorems\ 2.2, 2.3]{MR2207807},
\cite[Theorem\ 1]{MR2577332})}.

\end{theorem}

For other properties of $\varphi_{T},\psi_{T}$ we refer to \cite{MR2207807,MR2389004,MR2453098,MR2594359,MR2577332}.

\strut

The following set properties are frequently used in the sequel; for
$M,N\subset X\times X^{*}$, $V,W\subset X$
\begin{itemize}
\item $M\cap(V\times X^{*})\subset N$ $\Rightarrow$ $(\operatorname*{Pr}_{X}M)\cap V\subset\operatorname*{Pr}_{X}N$;
\item $M\cap(V\times X^{*})\subset W\times X^{*}$ $\Leftrightarrow$ $(\operatorname*{Pr}_{X}M)\cap V\subset W$;
and
\item $\operatorname*{Pr}_{X}(M\cap(V\times X^{*}))=(\operatorname*{Pr}_{X}M)\cap V$.
\end{itemize}
Throughout this article the conventions $\sup\emptyset=-\infty$ and
$\inf\emptyset=\infty$ are observed.

\eject

\eject

\section{Definitions and properties}

\begin{definition} \label{def-ident}\emph{ Let $X$ be a LCS. A
subset $V\subset X$ }identifies\emph{ $T:X\rightrightarrows X^{*}$
or $T$ }is identified by\emph{ $V$ if $[\varphi_{T|_{V}}\le c]\cap V\times X^{*}\subset\operatorname*{Graph}T$.
Equivalently, $V$ identifies $T$ iff every $z=(x,x^{*})\in V\times X^{*}$
that is m.r. to $T|_{V}$ belongs to $T$. Here $T|_{V}:X\rightrightarrows X^{*}$
is defined by $(x,x^{*})\in T|_{V}$ if $x\in V$ and $(x,x^{*})\in T$
or $\operatorname*{Graph}T|_{V}=\operatorname*{Graph}T\cap(V\times X^{*})$.}
\end{definition}

Note that the empty set identifies any operator, but if a non-empty
$V\subset X$ identifies $T$ then $V\cap D(T)\neq\emptyset$. Indeed,
if $V\neq\emptyset$ and $V\cap D(T)=\emptyset$ then $\varphi_{T|_{V}}=-\infty$,
$[\varphi_{T|_{V}}\le c]\cap V\times X^{*}=V\times X^{*}\not\subset\operatorname*{Graph}T$,
that is, $V$ does not identify $T$. Hence this notion is interesting
only when $V\cap D(T)\neq\emptyset$. 

When $T$ is non-void monotone, $V$ identifies $T$ iff $T|_{V}$
is maximal monotone in $V\times X^{*}$, that is, $\operatorname*{Graph}T|_{V}$
has no proper monotone extension in $V\times X^{*}$ or \emph{$[\varphi_{T|_{V}}\le c]\cap V\times X^{*}=\operatorname*{Graph}(T|_{V})$}. 

\medskip

In the context of a Banach space $X$, a monotone operator $T$ is
called of type (FPV) or maximal monotone locally (notion first introduced
in \cite{MR1249266} and further studied in \cite{MR1333366}) if
for every open convex $V\subset X$ either $V\cap D(T)=\emptyset$
or $V$ identifies $T$. For the sake of language, notation simplicity,
and notion uniformity we introduce the terminology\emph{ }identifiable
as an extension of the type (FPV) notion to a general operator in
the context of locally convex spaces. 

\begin{definition} \label{def-identifiable}\emph{ Let $(X,\tau)$
be a LCS. An operator $T:X\rightrightarrows X^{*}$ is }($\tau-$)identifiable\emph{
if $T$ is identified by every ($\tau-$)open convex subset $V$ of
$X$ such that $V\cap D(T)\neq\emptyset$.} \end{definition}

Note that $X$ identifies a monotone operator $T$ iff $T$ is maximal
monotone. Therefore every identifiable monotone operator is maximal
monotone.

\medskip

The identifiability of a maximal monotone operator $T$ is interesting
only on sets $V$ with $D(T)\not\subset V$ since $T\in\mathfrak{M}(X)$
is identified by every $V$ that contains $D(T)$.

The $\tau-$identifiability of an operator depends explicitly on the
topology $\tau$ and not only on the duality $(X,X^{*})$. 

The identifiability notion unifies several other notions from the
literature. For example for $X$ a Banach space, $T:X\rightrightarrows X^{*}$
is \emph{locally maximal monotone} (notion introduced in \cite[p.\ 583]{MR1191009})
iff $T$ is monotone and $T^{-1}:X^{*}\rightrightarrows X$ is identified
by every norm-open convex subset of $X^{*}$. 

The identifiability of a monotone operator $T$ is intrinsically related
to the sum theorem. More precisely, if $T+N_{C}$ is maximal monotone,
for every $C\subset X$ closed convex with $D(T)\cap\operatorname*{int}C\neq\emptyset$
then $T$ is identifiable. Under a Banach space settings this implication
is known for some time (see e.g. \cite[Proposition 3.3]{MR1333366})
but it also holds in a locally convex space context. 

The sum conjecture {[}SC{]} is true in reflexive Banach spaces (see
e.g. \cite[Theorem 1(a), p. 76]{MR0282272}). Therefore every maximal
monotone operator in a reflexive Banach space is identifiable.

\medskip

The class of open convex sets arises naturally in the identification
of maximal monotone operators. That is not the case for the class
of closed convex sets (even when they have non-empty interiors). Assume
that $X$ is a LCS and the closed convex $C\subset X$ identifies
$T\in\mathcal{M}(X)$. Since $T|_{C}\subset T+N_{C}\in\mathcal{M}(X)$
and $D(T+N_{C})\subset C$ one gets $T|_{C}=T+N_{C}\subset T$. The
contrapositive form of this fact shows that a closed convex $C\subset X$
does not identify $T\in\mathcal{M}(X)$ if $T+N_{C}\in\mathfrak{M}(X)$
(which happens for example when $X$ is a reflexive Banach space and
$D(T)\cap\operatorname*{int}C\neq\emptyset$) and $D(T)\not\subset C$.
However, in general, $C$ can identify $T+N_{C}$ (see Theorem \ref{C-amt}
below) and that leads to our next notion. 

\begin{definition} \label{def-local}\emph{ Let $X$ be a LCS. A
subset $V\subset X$} locates\emph{ $T:X\rightrightarrows X^{*}$
}in\emph{ $S\subset X$ or $T$ }is located by\emph{ $V$ in $S$
if $\operatorname*{Pr}_{X}[\varphi_{T|_{V}}\le c]\cap V\subset S$
(or $[\varphi_{T|_{V}}\le c]\cap V\times X^{*}\subset S\times X^{*}$).
Equivalently, $V$ locates $T$ in $S$ iff every $z=(x,x^{*})\in V\times X^{*}$
that is m.r. to $T|_{V}$ has $x\in S$.}

\emph{An operator $T:X\rightrightarrows X^{*}$ is called }locatable
in\emph{ $S$ iff every open convex subset $V$ of $X$ such that
$V\cap D(T)\neq\emptyset$ locates $T$ in $S$. When $S=D(T)$ we
simply say that $V$ }locates\emph{ $T$ or $T$ }is located by\emph{
$V$ and that $T$ is }locatable\emph{.} \end{definition}

As previously seen, his notion is interesting only when $V\cap D(T)\neq\emptyset$
because if $V\cap D(T)=\emptyset$ then $\operatorname*{Pr}_{X}[\varphi_{T|_{V}}\le c]\cap V\subset S$
reduces to $V\subset S$. 

\medskip

A monotone operator $T$ is locatable if for every open convex set
$V\subset X$ with $V\cap D(T)\neq\emptyset$, $T|_{V}$ cannot be
extended outside $D(T)\cap V$, as a monotone operator in $V\times X^{*}$. 

\medskip

In the literature, for $X$ a Banach space, an operator $T$ is called
of type weak-FPV (notion first introduced in \cite{astfnc}) if every
open convex subset $V\subset X$ with $V\cap D(T)\neq\emptyset$ locates
$T$. The terminology locatable is used as an extension and a simplified
notation of the type weak-FPV notion to the general locally convex
space settings. 

\medskip

Every locatable operator in $S$ is locatable in $S'$, whenever $S\subset S'$.
Also, every identifiable operator is locatable because $V$ locates
$T$ whenever $V$ identifies $T$. However, there exist monotone
operators that are locatable but not identifiable. Take for example
$T=\{0\}\times(X^{*}\setminus\{0\})$ where $X$ is a LCS. Note that
$T$ is not identifiable since it is not maximal monotone. By a direct
verification $T$ is locatable. Indeed, if $V$ is open convex with
$0\in V$ and $z=(x,x^{*})\in V\times X^{*}$ is m.r. to $T|_{V}=T$
then $x=0$ because $\{0\}\times X^{*}$ is the unique maximal monotone
extension of $T$.

We will see later, in Theorem \ref{lo+re <-> id} below, that for
a maximal monotone operator in the general context of a locally convex
space the locatable and identifiable notions coincide. 

\medskip

An operator $T:X\rightrightarrows X^{*}$ is automatically located
by every $V\subset D(T)$. Therefore the location of an operator $T$
is interesting only on sets $V\not\subset D(T)$. 

\begin{lemma} \label{X-locate}Let $X$ be a LCS and let $T:X\rightrightarrows X^{*}$.
Then $X$ locates $T$ iff $\varphi_{T}\ge c$ and $\operatorname*{Pr}_{X}[\varphi_{T}=c]\subset D(T)$.
If, in addition, $T\in\mathcal{M}(X)$ then $X$ locates $T$ iff
$\varphi_{T}\ge c$ and $\operatorname*{Pr}_{X}[\varphi_{T}=c]=D(T)$.
\end{lemma}

\begin{proof} Condition $X$ locates $T$ comes to $[\varphi_{T}\le c]\subset D(T)\times X^{*}$.
The conclusion follows after we take Theorem \ref{propr}  (i), (ii)
into account, that is, for every $T:X\rightrightarrows X^{*}$, $D(T)\times X^{*}\subset[\varphi_{T}\ge c]$
and that $T\subset[\varphi_{T}=c]$ whenever $T\in\mathcal{M}(X)$.
\end{proof}

\strut

Therefore the localization of an operator $T$ by $X$ depends on
the condition $\varphi_{T}\ge c$ also known as $T$ is of type negative-infimum
(NI for short); notion that was first used in \cite{MR2207807} and
introduced in \cite[Remark.\ 3.5]{tscr-arxiv,MR2453098}. Therefore,
according to Lemma \ref{X-locate}, every maximal monotone operator
is NI because $T\in\mathfrak{M}(X)$ $\Leftrightarrow$ $X$ identifies
$T$ $\Rightarrow$ $X$ locates $T$. 

It must be said that our NI notion differs fundamentally from the
NI notion introduced, for $X$ a Banach space, in $X^{*}\times X^{**}$
by Simons (see e.g. \cite[Definition 25.5, p. 99]{MR1723737}). The
NI-operators in the sense of Simons coincide with those of dense-type
in the sense of Gossez introduced in \cite{MR0313890} (see \cite{MR2731293}).
Since the dense-type property is stronger and has been introduced
prior to the NI class in the sense of Simons it is our opinion that
the use of NI notion in the sense of Simons is obsolete. Another essential
difference between these two notions, besides the underlying space
context, is that every maximal monotone operator is NI in the current
sense while not every maximal monotone operator is of dense-type (see
e.g. \cite[p. 89]{MR0298492}). For more explanations on comparing
these notions see \cite[p. 33]{MR2594359}\emph{ }and\emph{ }\cite[p.\ 662]{astfnc}\emph{.}

Since every $X-$locatable operator is NI, we expect in general that
the localization property depend on a localized NI type condition.

\begin{definition} \label{V-NI-D} \emph{Let $X$ be a LCS. An operator
$T:X\rightrightarrows X^{*}$ is of }negative-infimum type on\emph{
$V\subset X$ or simply $V-$NI if $V\times X^{*}\subset[\varphi_{T|_{V}}\ge c]$.
Equivalently, $T$ is $V-$NI iff $[\varphi_{T|_{V}}<c]\cap V\times X^{*}=\emptyset$
iff $\operatorname*{Pr}\!\!\,_{X}[\varphi_{T|_{V}}<c]\cap V=\emptyset$. }

\emph{The operator $T:X\rightrightarrows X^{*}$ is called }locally-NI\emph{
if, for every open convex $V\subset X$ such that $V\cap D(T)\neq\emptyset$,
$T$ is $V-$NI. }\end{definition}

Note that every operator is $\emptyset-$NI and if, for a certain
$V\neq\emptyset$, $T$ is $V-$NI then $V\cap D(T)\neq\emptyset$,
because $V\cap D(T)=\emptyset$ implies $\varphi_{T|_{V}}=-\infty$.
Notice also that if $T$ is $V-$NI then $V\subset\operatorname*{Pr}_{X}[\varphi_{T|_{V}}\ge c]$
while the converse is not true. The $X-$NI type coincides with the
NI type discussed above.

\begin{theorem} \label{V-NI-T}Let $X$ be a LCS, let $T:X\rightrightarrows X^{*}$,
and let $V\subset X$ be such that $V\cap D(T)\neq\emptyset$. The
following are equivalent

\medskip

\begin{tabular}{ll}
\emph{(i)} $T$ is $V-$NI, & \emph{(iv)} $[\varphi_{T|_{V}}\le c]\cap V\times X^{*}\subset[\varphi_{T|_{V}}\ge c]$,\tabularnewline
\emph{(ii)} $\varphi_{T|_{V}}\ge c$ in $V\times X^{*}$, & \emph{(v)} $[\varphi_{T|_{V}}<c]\cap V\times X^{*}\subset[\varphi_{T|_{V}}\ge c]$.\tabularnewline
\emph{(iii)} $\varphi_{T|_{V}}+\iota_{V\times X^{*}}\ge c$, & \tabularnewline
\end{tabular} \end{theorem}

\begin{proof} Left to the reader. \end{proof} 

\begin{theorem} \label{convex closure} Let $X$ be a LCS and let
$T\in\mathcal{M}(X)$. Assume that
\begin{equation}
V{\rm \ open\ convex},\ V\cap D(T)\neq\emptyset\ \Longrightarrow\ \operatorname*{Pr}\,\!\!_{X}[\varphi_{T|_{V}}<c]\cap V\subset\overline{D(T)}.\label{C}
\end{equation}
Then $\overline{D(T)}$ is convex. 

In particular, if $T\in\mathcal{M}(X)$ is locatable (in $\overline{D(T)}$)
or $T\in\mathcal{M}(X)$ is locally-NI then $\overline{D(T)}$ convex.
\end{theorem}

\begin{proof} Assume that $\overline{D(T)}$ is not convex. There
exist $x_{0},x_{1}\in D(T)$, $0<\rho<1$, and $U\in\mathscr{V}(0)$
such that $(x_{\rho}+U)\cap D(T)=\emptyset$, where $x_{t}:=tx_{1}+(1-t)x_{0}$,
$0\le t\le1$; in particular $x_{0}\neq x_{1}$. Let $x_{0}^{*}\in Tx_{0}$,
$x_{1}^{*}\in Tx_{1}$, and denote by $z_{0}=(x_{0},x_{0}^{*})$,
$z_{1}=(x_{1},x_{1}^{*})$, $z_{\rho}:=\rho z_{1}+(1-\rho)z_{0}$. 

Let $V\in\mathscr{V}(0)$ be balanced and $\gamma>0$ be such that
$V+V\subset U$ and $x_{1}-x_{0}\in\gamma V$. Take $y^{*}\in X^{*}$
with $\langle x_{0}-x_{1},y^{*}\rangle\ge\gamma(\rho(1-\rho)c(z_{1}-z_{0})+1)>0$
and $W\in\mathscr{V}(0)$ be open convex such that $W\subset V$ and
$\alpha:=\sup\{|\langle u,y^{*}\rangle|\mid u\in W\}$ is finite and
positive (see e.g. \cite[Theorem\ 1.18,\ p. 15]{MR1157815}). Let
$0<\epsilon<\min\{1,1/\alpha\}$. Let $\bar{z}_{\rho}:=z_{\rho}+(0,y^{*})$.

For every $z=(x,x^{*})\in T$ with $x\in D(T)\cap([x_{1},x_{\rho}]+\epsilon W)$
we have $x-x_{\rho}\not\in U$ and $x-x_{\lambda}\in\epsilon W$ for
some $1\ge\lambda>\rho$. This yields that $x_{\lambda}-x_{\rho}=(\lambda-\rho)(x_{1}-x_{0})\not\in V$.
Hence $(\lambda-\rho)\ge1/\gamma$. Since $x_{\rho}-x=(\lambda-\rho)(x_{0}-x_{1})+x_{\lambda}-x$
and $T\in\mathcal{M}(X)$ we have 
\[
c(\bar{z}_{\rho}-z)=c(z_{\rho}-z)+\langle x_{\rho}-x,y^{*}\rangle=\rho c(z_{1}-z)+(1-\rho)c(z_{0}-z)-\rho(1-\rho)c(z_{1}-z_{0})+\langle x_{\rho}-x,y^{*}\rangle
\]
\[
\ge(\lambda-\rho)\langle x_{0}-x_{1},y^{*}\rangle+\langle x_{\lambda}-x,y^{*}\rangle-\rho(1-\rho)c(z_{1}-z_{0})\ge\tfrac{1}{\gamma}\langle x_{0}-x_{1},y^{*}\rangle-\alpha\epsilon-\rho(1-\rho)c(z_{1}-z_{0})>0.
\]
This yields $\bar{z}_{\rho}\in[\varphi_{T|_{[x_{1},x_{\rho}]+\epsilon W}}<c]\cap([x_{1},x_{\rho}]+\epsilon W)\subset\overline{D(T)}\times X^{*}$
and the contradiction $x_{\rho}\in\overline{D(T)}$. \end{proof}

\begin{remark}[$T$ is $V-$NI versus $T|_V$ is NI] \emph{First note
that $T$ is $V-$NI whenever $V\subset D(T)$ since in this case
$V\times X^{*}=D(T|_{V})\times X^{*}\subset[\varphi_{T|_{V}}\ge c]$
(see Theorem \ref{propr}  (i)) or because, in this case, $V$ locates
$T$ (see Theorem \ref{V-locate} below). Also, it is straightforward
that $T$ is $V-$NI whenever $T|_{V}$ is $(X-)$NI. The converse
of this fact, namely, whether $T|_{V}$ is NI whenever $T$ is $V-$NI
fails to be true in any LCS $X$ even when $T$ is maximal monotone,
$V$ is convex, and $V$ is open or closed with empty or non-empty
interior.}

\emph{We base our following examples on the fact that every monotone
NI operator admits a unique maximal monotone extension (see Theorem
\ref{VNIR} below or \cite[Proposition 4 (iii)]{MR2594359}).}

\emph{For example, for every $T\in\mathcal{M}(X)$ with a non-singleton
domain and for every $x\in D(T)$, $T$ is $\{x\}-$NI while $T|_{\{x\}}=\{x\}\times Tx$
is not NI because $\{x\}\times X$ and any maximal monotone extension
of $T$ are different maximal monotone extensions of $T|_{\{x\}}$.}

\emph{Similar considerations can be made for a non-NI operator $T$
which is $D(T)-$NI but $T|_{D(T)}=T$ is not NI. }

\emph{Let $C\varsubsetneq X$ be closed convex with $\operatorname*{int}C\neq\emptyset$.
Then $N_{C}$ is $\operatorname*{int}C-$NI (this fact can also be
checked directly from $N_{C}|_{\operatorname*{int}C}=\operatorname*{int}C\times\{0\}$
and $\varphi_{\operatorname*{int}C\times\{0\}}(x,x^{*})=\sigma_{C}(x^{*})$,
$(x,x^{*})\in X\times X^{*}$). But $N_{C}|_{\operatorname*{int}C}$
is not NI since it admits two distinct (maximal) monotone extensions:
$N_{C}$ and $X\times\{0\}$. Similarly, for every closed convex set
$D\subset\operatorname*{int}C$ (with possible empty interior) we
have that $N_{C}$ is $D-$NI and $N_{C}|_{D}$ is not NI because
$N_{C}|_{\operatorname*{int}C}$ is not NI.} \end{remark} 

\begin{remark}[The NI method] \emph{\label{NI method} In general
it is hard to verify the NI condition directly, even when $T$ is
monotone, since the closed forms of $\varphi_{T}$, $\psi_{T}$ are
known only for few types of operators (see e.g. \cite{MR2291550,MR2594359})
and, when $X$ is a non-reflexive Banach space, the coupling $c$
is not continuous with respect to any topology on $X\times X^{*}$
compatible with the natural duality $(X\times X^{*},X^{*}\times X)$
(see \cite[Appendix]{MR2577332}).}

\emph{Given a LCS $X$, the first direct method to prove that an operator
$T:X\rightrightarrows X^{*}$ is of NI type has been developed in
\cite[Theorem.\ 1.1]{MR2207807} and is summarized as follows
\begin{equation}
(z=(x,x^{*})\ {\rm is\ m.r.\ to}\ T\Rightarrow x\in D(T))\Longrightarrow T\ {\rm is\ NI},\label{NI-method}
\end{equation}
or, equivalently, $X$ locates $T$ $\Longrightarrow T\ {\rm is\ NI}$.
The reader recognizes that this NI method is contained in Lemma \ref{X-locate}
and that its converse holds under the additional condition $\operatorname*{Pr}_{X}[\varphi_{T}=c]\subset D(T)$.}

\emph{The following is a slightly improved version of (\ref{NI-method}),
namely 
\begin{equation}
(\operatorname*{Pr}\!\!\,_{X}[\varphi_{T}<c]\subset D(T))\Longrightarrow T\ {\rm is\ NI}.\label{NI-met-v2}
\end{equation}
Indeed, if $z=(x,x^{*})\in[\varphi_{T}<c]$ then $x\in\operatorname*{Pr}\!\!\,_{X}[\varphi_{T}<c]\subset D(T)$;
whence, according to Theorem \ref{propr} (i), $z\in[\varphi_{T}\ge c]$
a contradiction. Therefore $[\varphi_{T}<c]$ is empty, that is, $T$
is NI.} \end{remark}

Similar considerations for a $V-$NI method are contained in the following
result.

\begin{theorem}[The $V-$NI method]\label{V-NI method} Let $X$ be
a LCS, let $T:X\rightrightarrows X^{*}$, and let $V\subset X$. Then
$T$ is $V-$NI iff $\operatorname*{Pr}_{X}[\varphi_{T|_{V}}<c]\cap V\subset D(T)$.
\end{theorem}

\begin{proof} It suffices to note that, in general, $\operatorname*{Pr}_{X}[\varphi_{T|_{V}}<c]\cap V\cap D(T)=\emptyset$
due to $D(T|_{V})\times X^{*}=(D(T)\cap V)\times X^{*}\subset[\varphi_{T|_{V}}\ge c]$.
\end{proof}

\begin{theorem} \label{V-locate}Let $X$ be a LCS, let $T:X\rightrightarrows X^{*}$,
and let $V\subset X$ be such that $V\cap D(T)\neq\emptyset$. The
following are equivalent

\medskip

\emph{(i)} $V$ locates $T$,

\medskip

\emph{(ii)} $T$ is $V-$NI and $\operatorname*{Pr}_{X}[\varphi_{T|_{V}}=c]\cap V\subset D(T)$,

\medskip

\emph{(iii)} $\operatorname*{Pr}_{X}\operatorname*{dom}\varphi_{T|_{V}}\cap V\subset\operatorname*{Pr}_{X}([\varphi_{T|_{V}}\ge c]\cap\operatorname*{dom}\varphi_{T|_{V}})$
and $\operatorname*{Pr}_{X}[\varphi_{T|_{V}}=c]\cap V\subset D(T)$.

\medskip

If, in addition, $T|_{V}\in\mathcal{M}(X)$ then $V$ locates $T$
iff $T$ is $V-$NI and $\operatorname*{Pr}_{X}[\varphi_{T|_{V}}=c]\cap V=D(T)\cap V$.

\end{theorem}

\begin{proof} (i) $\Rightarrow$ (ii) Recall that $V$ locates $T$
means $[\varphi_{T|_{V}}\le c]\cap V\times X^{*}\subset D(T)\times X^{*}$
from which $[\varphi_{T|_{V}}\le c]\cap V\times X^{*}\subset D(T|_{V})\times X^{*}\subset[\varphi_{T|_{V}}\ge c]$;
whence $T$ is $V-$NI and $\operatorname*{Pr}_{X}[\varphi_{T|_{V}}\le c]\cap V=\operatorname*{Pr}_{X}[\varphi_{T|_{V}}=c]\cap V\subset D(T)$.

(ii) $\Rightarrow$ (iii) Because $T$ is $V-$NI we have $V\times X^{*}\subset[\varphi_{T|_{V}}\ge c]$
followed by $\operatorname*{dom}\varphi_{T|_{V}}\cap V\times X^{*}\subset\operatorname*{dom}\varphi_{T|_{V}}\cap[\varphi_{T|_{V}}\ge c]$.

(iii) $\Rightarrow$ (i) Let $x\in\operatorname*{Pr}_{X}[\varphi_{T|_{V}}\le c]\cap V$.
Take $x_{1}^{*}\in X^{*}$ such that $(x,x_{1}^{*})\in[\varphi_{T|_{V}}\le c]\subset\operatorname*{dom}\varphi_{T|_{V}}$.
Then $x\in\operatorname*{Pr}_{X}(\operatorname*{dom}\varphi_{T|_{V}})\cap V\subset\operatorname*{Pr}_{X}([\varphi_{T|_{V}}\ge c]\cap\operatorname*{dom}\varphi_{T|_{V}})$,
so $(x,x_{2}^{*})\in[\varphi_{T|_{V}}\ge c]\cap\operatorname*{dom}\varphi_{T|_{V}}$
for some $x_{2}^{*}\in X^{*}$. The function $f:[0,1]\rightarrow\mathbb{R},$
$f(t)=(\varphi_{T|_{V}}-c)(x,tx_{1}^{*}+(1-t)x_{2}^{*})$ is continuous
and $f(0)\ge0$, $f(1)\le0$. Therefore there is $s\in[0,1]$ such
that $f(s)=0$, that is, $(x,sx_{1}^{*}+(1-s)x_{2}^{*})\in[\varphi_{T|_{V}}=c]$.
Therefore $x\in\operatorname*{Pr}_{X}[\varphi_{T|_{V}}=c]\cap V\subset D(T)$
and so $V$ locates $T$. 

If in addition $T|_{V}\in\mathcal{M}(X)$ then $T|_{V}\subset[\varphi_{T|_{V}}=c]$,
$D(T)\cap V=D(T|_{V})\subset\operatorname*{Pr}_{X}[\varphi_{T|_{V}}=c]$,
and the last part of the conclusion follows from (i) $\Leftrightarrow$
(ii). \end{proof}

\strut

In Theorem \ref{V-locate} we saw that $T$ being of $V-$NI type
is an important part of $V$ locating $T$ and as a consequence every
locatable operator is locally-NI. As previously stated, the other
condition in Theorem \ref{V-locate}, namely, $\operatorname*{Pr}_{X}[\varphi_{T|_{V}}=c]\cap V\subset D(T)$
is hard to verify directly due to the unwieldy nature of $\varphi_{T}$.
Fortunately, this latter condition can be replaced by representability. 

\begin{definition} \emph{Let $(X,\tau)$ be a LCS. An operator }$T:X\rightrightarrows X^{*}$\emph{
is }representable\emph{ }in\emph{ $V\subset X$ or $V-$}representable\emph{
if $V\cap D(T)\ne\emptyset$ and there is $h\in\mathscr{R}$ (that
is, $h\ge c$ and $h\in\Gamma_{\tau\times w^{*}}(X\times X^{*})$)
such that $[h=c]\cap V\times X^{*}=\operatorname*{Graph}(T|_{V})$.
The function $h$ is called a $V-$}representative\emph{ }of\emph{
$T$. The class of }$V-$\emph{representatives} \emph{of} $T$\emph{
is denoted by $\mathscr{{R}}_{T}^{V}$.} \end{definition} 

As previously seen, the condition $V\cap D(T)\neq\emptyset$ can be
avoided but its presence makes the previous definition meaningful. 

An $X-$representable operator $T:X\rightrightarrows X^{*}$ is simply
called \emph{representable} and the class of its representatives is
denoted by $\mathscr{{R}}_{T}$, notion that was first considered
in this form in \cite{MR2207807}. For properties of representable
operators see \cite{MR2389004,MR2453098,MR2583911,MR2594359,MR2577332}.

In other words, $T$ is $V-$representable if $T|_{V}$ is the trace
of the representable operator $[h=c]$ on $V\times X^{*}$, where
$h\in\mathscr{{R}}$. 

\begin{remark} \emph{\label{reprez}Note that }
\begin{itemize}
\item \emph{$T|_{V}$ is monotone whenever $T$ is $V-$representable since
$[h=c]\in\mathcal{M}(X)$ for every $h\in\mathscr{C}$ (see e.g. \cite[Proposition 4]{MR2086060}
or \cite[Lemma\ 3.1]{MR2453098}); }
\item \emph{$T$ is $W-$representable whenever $T$ is $V-$representable
and $V\supset W$; in this case every $V-$representative is a $W-$representative
of $T$, that is, $\mathscr{{R}}_{T}^{V}\subset\mathscr{{R}}_{T}^{W}$.
In particular, if $T$ is representable then, for every $V\subset X$,
$T$ is $V-$representable. Conversely, if $T$ is $V-$representable
then $T$ need not be representable because we can modify $T$ outside
$V\times X^{*}$. For example, for every $V\varsubsetneq X$, $x\not\in V$,
$T=(X\setminus\{x\})\times\{0\}$ is $V-$representable since $\varphi_{T}=\psi_{T}=\iota_{X\times\{0\}}\in\mathscr{{R}}_{T}^{V}$
and $T$ is not representable because $T$ is not closed or because
$T\varsubsetneq[\psi_{T}=c]$ (see \cite[Theorem\ 2.2]{MR2207807}
or Theorem \ref{V-repres} below); }
\item \emph{If $T|_{V}$ is representable then $T$ is $V-$representable;
in this case every representative of $T|_{V}$ is a $V-$representative
of $T$, i.e., $\mathscr{{R}}_{T|_{V}}\subset\mathscr{{R}}_{T}^{V}$.
Indeed, if $h\in\mathscr{R}$ has $T|_{V}=[h=c]$ then }$V\times X^{*}\cap[h=c]=T|_{V}$.
\emph{Conversely, if $T$ is $V-$representable with $h$ a $V-$representative
of $T$ such that $\operatorname*{Pr}_{X}[h=c]\subset V$ then $T|_{V}$
is representable with representative $h$. In general, without the
additional condition, the converse is not true, in any LCS $X$, even
if we work with $T\in\mathfrak{M}(X)$ and $V$ open convex. Indeed,
take $C\varsubsetneq X$ closed convex with $\operatorname*{int}C\neq\emptyset$,
$T=N_{C}$, $V=\operatorname*{int}C$. Then $T$ is representable,
since it is maximal monotone (see e.g. \cite[Theorem\ 2.3]{MR2207807}
or Theorem \ref{X-amt} below), while $T|_{V}=\operatorname*{int}C\times\{0\}$
is not, for example because $\psi_{T|_{V}}=\iota_{C\times\{0\}}$
and $T|_{V}=\operatorname*{int}C\times\{0\}\subsetneq C\times\{0\}=[\psi_{T|_{V}}=c]$
(see \cite[Theorem 2.2]{MR2207807} or Theorem \ref{V-repres} below);}
\item \emph{However, when $C$ is closed convex, $T$ is $C-$representable
iff $T|_{C}$ is representable. Indeed, if $h$ is a $C-$representative
of $T$ then $h+\iota_{C\times X^{*}}$ is a representative of $T|_{C}$. }
\item \emph{An operator $T\in\mathcal{M}(X)$ is $D(T)-$representable whenever
$D(T)$ identifies $T$. Indeed, let $\overline{T}$ be a representable
extension of $T$ (such as $[\psi_{T}=c]$), and let $h\in\mathscr{R}_{\overline{T}}$,
in particular, $\overline{T}=[h=c]$. Hence $[h=c]\cap D(T)\times X^{*}=\overline{T}|_{D(T)}=T$
because $T$ is maximal monotone in $D(T)\times X^{*}$.}
\end{itemize}
\end{remark}

The following result is a generalization of \cite[Theorem\ 2.2]{MR2207807},
\cite[Theorem\ 1(ii)]{MR2577332}.

\begin{theorem} \label{V-repres}Let $X$ be a LCS, let $T:X\rightrightarrows X^{*}$,
and let $V\subset X$ be such that $V\cap D(T)\neq\emptyset$. The
following are equivalent

\medskip

\emph{(i)} $T$ is $V-$representable,

\medskip

\emph{(ii)} $T|_{V}\in\mathcal{M}(X)$ and $[\psi_{T|_{V}}=c]\cap V\times X^{*}\subset T|_{V}$,

\medskip

\emph{(iii)} $T|_{V}\in\mathcal{M}(X)$ and $[\psi_{T|_{V}}=c]\cap V\times X^{*}=T|_{V}$,

\medskip

\emph{(iv)} $\psi_{T|_{V}}$ is a $V-$representative of $T|_{V}$,
i.e., $\psi_{T|_{V}}\in\mathscr{{R}}_{T|_{V}}^{V}$. \end{theorem}

\begin{proof} (i) $\Rightarrow$ (ii) Let $h\in\mathscr{R}$ be such
that $[h=c]\cap V\times X^{*}=T|_{V}$. Then $h\le c_{T|_{V}}$ followed
by $c\le h\le\psi_{T|_{V}}$ since $h\in\mathscr{R}$. Therefore $[\psi_{T|_{V}}=c]\subset[h=c]$
and so $[\psi_{T|_{V}}=c]\cap V\times X^{*}\subset T|_{V}$. 

(ii) $\Rightarrow$ (iii) From $T|_{V}\in\mathcal{M}(X)$ we know
that $T|_{V}\subset[\psi_{T|_{V}}=c]$ (see \cite[Proposition\ 3.2 (viii)]{tscr-arxiv,MR2389004}
or \cite[(9)]{MR2577332}). 

For (iii) $\Rightarrow$ (iv) it suffices to notice that $\psi_{T|_{V}}\in\mathscr{R}$,
since $T|_{V}\in\mathcal{M}(X)$ (see Theorem \ref{propr}(ii)).

The implication (iv) $\Rightarrow$ (i) is trivial. \end{proof}

\begin{remark} \emph{In case $V$ is closed convex we have $\operatorname*{dom}\psi_{T|_{V}}\subset V\times X^{*}$
so $[\psi_{T|_{V}}=c]\subset V\times X^{*}$, $[\psi_{T|_{V}}=c]\cap V\times X^{*}=[\psi_{T|_{V}}=c]$,
and Theorem \ref{V-repres} says again that $T$ is $V-$representable
iff $T|_{V}$ is representable.} \end{remark}

\begin{remark} \emph{Given $X$ a LCS, $T:X\rightrightarrows X^{*}$,
and $V\subset X$ such that $T|_{V}\in\mathcal{M}(X)$, the operator
$R:=[\psi_{T|_{V}}=c]\cap V\times X^{*}$ is the smallest $V-$representable
extension of $T|_{V}$ in $V\times X^{*}$. Indeed, for every $h\in\mathscr{R}$
such that $[h=c]\cap V\times X^{*}\supset T|_{V}$ we have $h\le c_{T|_{V}}$,
$c\le h\le\psi_{T|_{V}}$ and so $[\psi_{T|_{V}}=c]\cap V\times X^{*}\subset[h=c]\cap V\times X^{*}$. }

\emph{Also,} $\varphi_{R}=\varphi_{T|_{V}}$, $\psi_{R}=\psi_{T|_{V}}$
\emph{because $T|_{V}\subset R\subset[\psi_{T|_{V}}=c]$ and $\varphi_{T|_{V}}=\varphi_{[\psi_{T|_{V}}=c]}$
(see \cite[Proposition\ 4,\ p.\ 35]{MR2594359}). }\end{remark}

\section{Characterizations}

\begin{theorem} \label{V-amt}Let $X$ be a LCS, let $T:X\rightrightarrows X^{*}$,
and let $V\subset X$ be such that $V\cap D(T)\neq\emptyset$. Consider
the conditions

\medskip

\emph{(i)} $T|_{V}\in\mathcal{M}(X)$ and $V$ identifies $T$,

\medskip

\emph{(ii)} $T$ is $V-$representable and $V$ locates $T$, 

\medskip

\emph{(iii)} $T$ is $V-$representable and $V-$NI. 

\medskip

Then \emph{(i) $\Rightarrow$ (ii) $\Rightarrow$ (iii)}. If, in addition,
$V$ is algebraically open then \emph{(i) $\Leftrightarrow$ (ii)
$\Leftrightarrow$ (iii)}. \end{theorem}

\begin{proof} We adapt the proof in \cite[Theorem\ 2.3]{MR2207807}
and we refer to \cite{MR2389004,MR2453098,MR2207807} for other different
arguments. The implication (ii) $\Rightarrow$ (iii) is contained
in Theorem \ref{V-locate}. 

(i) $\Rightarrow$ (ii) Since $V$ identifies it also locates $T$.
Hence, according to Theorem \ref{V-locate}, $T$ is $V-$NI. We know
that $\psi_{T|_{V}}\ge\max\{\varphi_{T|_{V}},c\}$ since $T|_{V}\in\mathcal{M}(X)$
(see \cite[Proposition\ 3.2 (vii)]{tscr-arxiv,MR2389004}). This yields
that $[\psi_{T|_{V}}=c]\cap V\times X^{*}\subset[\varphi_{T|_{V}}=c]\cap V\times X^{*}\subset T$
from which, according to Theorem \ref{V-repres}, it follows that
$T$ is $V-$representable. 

Assume that $V$ is algebraically open, i.e., $V=\operatorname*{core}V$. 

(iii) $\Rightarrow$ (i) Since $T$ is $V-$NI we have $[\varphi_{T|_{V}}\le c]\cap V\times X^{*}=[\varphi_{T|_{V}}=c]\cap V\times X^{*}$
and from $T$ being $V-$representable we know that $T|_{V}\in\mathcal{M}(X)$
so, according to Theorem \ref{propr} (ii), $\psi_{T|_{V}}\ge c$;
whence $[\psi_{T|_{V}}=c]\subset[\varphi_{T|_{V}}=c]$ (see Lemma
\ref{h>c}). This yields $T|_{V}=[\psi_{T|_{V}}=c]\cap V\times X^{*}\subset[\varphi_{T|_{V}}=c]\cap V\times X^{*}$
since $T$ is $V-$representable. To conclude it suffices to show
that 
\[
[\varphi_{T|_{V}}=c]\cap V\times X^{*}\subset[\psi_{T|_{V}}=c]\cap V\times X^{*}.
\]
Let $z\in[\varphi_{T|_{V}}=c]\cap V\times X^{*}$. Because $V$ is
algebraically open, for every $v\in X\times X^{*}$ there is $t_{v}>0$
such that $z+tv\in V\times X^{*}$, for every $0<t<t_{v}$. Hence,
since $T$ is $V-$NI, $\varphi_{T|_{V}}(z+tv)\ge c(z+tv)$, for every
$0<t<t_{v}$. The directional derivative of $\varphi_{T|_{V}}$ at
$z$ in the direction of $v$ satisfies
\[
\forall v\in Z,\ \varphi_{T|_{V}}'(z;v):=\lim_{t\downarrow0}\frac{\varphi_{T|_{V}}(z+tv)-\varphi_{T|_{V}}(z)}{t}\ge\lim_{t\downarrow0}\frac{c(z+tv)-c(z)}{t}=z\cdot v.
\]
This shows that $z\in\partial\varphi_{T|_{V}}(z)$, where ``$\partial$''
is considered under the natural duality $(Z,Z)$. Therefore $\psi_{T|_{V}}(z)+\varphi_{T|_{V}}(z)=z\cdot z=2c(z)$
which, together with $\varphi_{T|_{V}}(z)=c(z)$, implies that $\psi_{T|_{V}}(z)=c(z)$.
\end{proof}

In particular, for $V=X$ we recover the following maximal monotonicity
characterization.

\begin{theorem} \label{X-amt} \emph{(\cite[Theorem\ 2.3]{MR2207807},
\cite[Theorem\ 1 (ii)]{MR2577332})} Let $X$ be a LCS. Then $T\in\mathfrak{M}(X)$
iff $T$ is representable and NI. \end{theorem}

\begin{theorem} \label{identifiable} Let $X$ be a LCS, let $T:X\rightrightarrows X^{*}$,
and let $\mathcal{V}$ be a class of algebraically open subsets of
$X$ such that $X\in\mathcal{V}$. Then $T\in\mathcal{M}(X)$ and
$T$ is identified by every $V\in\mathcal{V}$ iff $T$ is representable
and, for every $V\in\mathcal{V}$, $T$ is $V-$NI. \end{theorem}

\begin{proof} Since every representable operator is $V-$representable
(for every $V\subset X$) and monotone the converse implication is
straightforward from Theorem \ref{V-amt}. For the direct implication
one gets that $T$ is maximal monotone (and implicitly representable)
because it is identified by $X\in\mathcal{V}$. Again Theorem \ref{V-amt}
completes the argument. \end{proof}

\begin{theorem} \label{VNIR} Let $X$ be a LCS, let $T:X\rightrightarrows X^{*}$,
and let $V\subset X$ be non-empty algebraically open and convex such
that $T|_{V}\in\mathcal{M}(X)$ and $T$ is $V-$NI. Then
\begin{equation}
[\psi_{T|_{V}}=c]\cap V\times X^{*}=[\varphi_{T|_{V}}=c]\cap V\times X^{*}=[\varphi_{T|_{V}}\le c]\cap V\times X^{*}\label{vnir}
\end{equation}
is the unique $V-$representable extension and the unique maximal
monotone extension in $V\times X^{*}$ of $T|_{V}$. 

If, in addition, $T\in\mathcal{M}(X)$ then the string of equalities
in (\ref{vnir}) can be completed to
\begin{equation}
[\psi_{T|_{V}}=c]\cap V\times X^{*}=[\varphi_{T}=c]\cap V\times X^{*}=[\psi_{T}=c]\cap V\times X^{*}.\label{vnirc}
\end{equation}

If, in addition, $T\in\mathcal{M}(X)$ and $T$ is $V-$representable
then $V$ identifies $T$ and 
\[
[\psi_{T|_{V}}=c]\cap V\times X^{*}=[\varphi_{T|_{V}}=c]\cap V\times X^{*}=[\varphi_{T|_{V}}\le c]\cap V\times X^{*}
\]
\begin{equation}
=[\varphi_{T}=c]\cap V\times X^{*}=[\psi_{T}=c]\cap V\times X^{*}=\operatorname*{Graph}(T|_{V}).\label{vn}
\end{equation}

\end{theorem}

\begin{proof} Let $R:=[\psi_{T|_{V}}=c]\cap V\times X^{*}$. Then
$R$ is $V-$representable and $V-$NI since $T|_{V}\in\mathcal{M}(X)$,
$T$ is $V-$NI, and $\varphi_{R}=\varphi_{T|_{V}}$. According to
Theorem \ref{V-amt}, $V$ identifies $R$, i.e., $R$ is maximal
monotone in $V\times X^{*}$. From $\psi_{T|_{V}}\ge\varphi_{T|_{V}}$
and $\varphi_{T|_{V}}\ge c$ in $V\times X^{*}$ we know that $R\subset[\varphi_{T|_{V}}=c]\cap V\times X^{*}=[\varphi_{T|_{V}}\le c]\cap V\times X^{*}$
so $R=[\varphi_{T|_{V}}=c]\cap V\times X^{*}$ since $[\varphi_{T|_{V}}=c]\cap V\times X^{*}\in\mathcal{M}(X)$.
Taking into consideration that $R$ is the smallest $V-$representable
extension of $T|_{V}$ the conclusion follows. 

If, in addition, $T\in\mathcal{M}(X)$ then, due to the facts that
$T$ is $V-$NI and $T|_{V}\subset T$, we have that for every $z\in V\times X^{*}$
\[
c(z)\le\varphi_{T|_{V}}(z)\le\varphi_{T}(z)\le\psi_{T}(z)\le\psi_{T|_{V}}(z),
\]
whence $[\psi_{T|_{V}}=c]\cap V\times X^{*}\subset[\psi_{T}=c]\cap V\times X^{*}\subset[\varphi_{T}=c]\cap V\times X^{*}\subset[\varphi_{T|_{V}}=c]\cap V\times X^{*}$.
Relation (\ref{vnir}) completes the proof of (\ref{vnirc}). 

Relation (\ref{vn}) follows from Theorems \ref{V-repres}, \ref{V-amt}
and relations (\ref{vnir}), (\ref{vnirc}). \end{proof} 

\strut

We are ready to prove that for a representable (and implicitly for
a maximal monotone) operator the locatable and identifiable notions
coincide.

\begin{theorem} \label{lo+re <-> id} Let $X$ be a LCS and let $T:X\rightrightarrows X^{*}$.
The following are equivalent

\medskip

\emph{(i)} $T\in\mathcal{M}(X)$ and $T$ is identifiable,

\medskip

\emph{(ii)} $T$ is representable and locatable, 

\medskip

\emph{(iii)} $T$ is representable and locally-NI.

\medskip

In particular, every representable and locatable operator is maximal
monotone. \end{theorem}

\begin{proof} (i) $\Leftrightarrow$ (iii) is a particular case of
Theorem \ref{identifiable} for $\mathcal{V}=\{V\subset X\mid V\ {\rm is\ open\ and\ convex},\ V\cap D(T)\neq\emptyset\}$. 

(i) $\Rightarrow$ (ii) is true since every identifiable monotone
operator is locatable and maximal monotone. 

(ii) $\Rightarrow$ (iii) is straightforward since every locatable
operator is locally-NI. \end{proof}

\strut

The global representability condition in the previous theorem can
be replaced by a weaker local form of it.

\begin{definition} \emph{Let $X$ be a LCS. An operator }$T:X\rightrightarrows X^{*}$\emph{
is }low-representable\emph{ if, for every $z=(x,x^{*})\in[\psi_{T}=c]$,
there is} $V\in\mathscr{{V}}(x)$ \emph{such that $T$ is $V-$representable.}
\end{definition} 

Every representable operator is low-representable (just take $V=X$). 

\begin{theorem} \label{LowRep} Let $X$ be a LCS and let $T:X\rightrightarrows X^{*}$.
Consider the conditions

\medskip

\emph{(iv)} $T$ is monotone, low-representable, and locatable;

\medskip

\emph{(v)} $T$ is monotone, low-representable, and locally-NI.

\medskip

Conditions \emph{(i) \textendash{} (iii)} being those from Theorem
\ref{lo+re <-> id}, we have \emph{(i) $\Leftrightarrow$ (ii) $\Leftrightarrow$
(iii) $\Leftrightarrow$ (iv) $\Leftrightarrow$ (v).} \end{theorem}

\begin{proof} The implications (ii) $\Rightarrow$ (iv), (iii) $\Rightarrow$
(v), (iv) $\Rightarrow$ (v) are plain. 

For (v) $\Rightarrow$ (iii) we prove that $T$ is representable,
i.e., $[\psi_{T}=c]\subset T$. For every $z=(x,x^{*})\in[\psi_{T}=c]$
let $V\in\mathscr{{V}}(x)$ be open convex and such that $T$ is $V-$representable.
Then, according to Theorem \ref{VNIR}, $z\in[\psi_{T}=c]\cap V\times X^{*}\subset T$.
\end{proof} 

\begin{theorem} \label{UI} Let $(X,\tau)$ be a LCS and let $T\in\mathcal{M}(X)$
be locally-NI. Then $[\varphi_{T}=c]=[\varphi_{T}\le c]=[\psi_{T}=c]$
is the unique identifiable extension of $T$. \end{theorem} 

\begin{proof} Because $T$ is NI, from (\ref{vnir}), $S:=[\varphi_{T}=c]=[\varphi_{T}\le c]=[\psi_{T}=c]$
is the unique maximal monotone extension of $T$. Since every identifiable
operator is maximal monotone, it suffices to prove that $S$ is locally-NI
to get that $S$ is the unique identifiable extension of $T$. 

But, if an open convex $V\subset X$ has $V\cap D(S)\neq\emptyset$
then $V\cap D(T)\neq\emptyset$. 

Indeed, if $x\in V\cap D(S)$ pick any $x^{*}\in S(x)$ and set $z=(x,x^{*})\in S$.
Then, since $T$ is $V-$NI, $c(z)\le\varphi_{T|_{V}}(z)\le\varphi_{T}(z)=c(z)$
so, according to Theorem \ref{VNIR}, $z\in[\varphi_{T|_{V}}=c]\cap V\times X^{*}=[\psi_{T|_{V}}=c]\cap V\times X^{*}\subset\operatorname*{dom}\psi_{T|_{V}}\subset\operatorname*{cl}_{\tau\times w^{*}}(\operatorname*{Graph}(T|_{V}))$
followed by $x\in\operatorname*{Pr}_{X}(\operatorname*{dom}\psi_{T|_{V}})\subset\overline{\operatorname*{conv}}(D(T)\cap V)\subset\overline{D(T)}$
due to the convexity of $\overline{D(T)}$ (see Theorem \ref{convex closure}).
Hence $x\in V\cap\overline{D(T)}\neq\emptyset$ from which $V\cap D(T)\neq\emptyset$
since $V$ is open. 

Hence, for every open convex $V\subset X$ such that $V\cap D(S)\neq\emptyset$,
$\varphi_{S|_{V}}\ge\varphi_{T|_{V}}\ge c$ in $V\times X^{*}$, because
$T$ is $V-$NI, i.e., $S$ is $V-$NI.Therefore $S$ is locally-NI.
\end{proof} 

\strut

The next result is a version of Theorem \ref{V-amt} for closed convex
sets with non-empty interior. First note that for every $T:X\rightrightarrows X^{*}$
and $C\subset X$ we have
\begin{equation}
[\varphi_{T+N_{C}}\le c]\cap C\times X^{*}=[\varphi_{T|_{C}}\le c]\cap C\times X^{*}.\label{tnc}
\end{equation}
Indeed, the direct inclusion follows from $T|_{C}\subset T+N_{C}$.
Conversely, if $z=(x,x^{*})\in C\times X^{*}$ is m.r. to $T|_{C}$
and $(a,a^{*})\in T|_{C}$, $n^{*}\in N_{C}(a)$ then $\langle x-a,n^{*}\rangle\le0$,
$\langle x-a,x^{*}-a^{*}\rangle\ge0$, and $\langle x-a,x^{*}-a^{*}-n^{*}\rangle\ge0$,
that is, $z$ is m.r. to $T+N_{C}$. 

\begin{theorem} \label{C-amt}Let $X$ be a LCS, let $T:X\rightrightarrows X^{*}$,
and let $C\subset X$ be closed convex such that $D(T)\cap\operatorname*{int}C\neq\emptyset$.
If $T$ is $C-$representable then the following are equivalent

\medskip

\emph{(i)} $C$ locates $T$,

\medskip

\emph{(ii)} $T$ is $C-$NI,

\medskip

\emph{(iii)} $[\varphi_{T|_{C}}\le c]\cap C\times X^{*}\subset T+N_{C}$,

\medskip

\emph{(iv)} $C$ identifies $T+N_{C}$. \end{theorem}

\begin{proof} The implication (i) $\Rightarrow$ (ii) is part of
Theorem \ref{V-locate} while (iii) $\Rightarrow$ (i) is plain.

(ii) $\Rightarrow$ (iii) Since $T$ is $C-$NI, we have $h:=\varphi_{T|_{C}}+\iota_{C\times X^{*}}\in\mathscr{R}$
and $[h=c]=[\varphi_{T|_{C}}=c]\cap C\times X^{*}=[\varphi_{T|_{C}}\le c]\cap C\times X^{*}$.
According to \cite[Theorem 2.8.7 (iii),\ p.\ 127]{MR1921556} 
\begin{equation}
h^{\square}(x,x^{*})=\min\{\psi_{T|_{C}}(x,u^{*})+\sigma_{C}(x^{*}-u^{*})\mid u^{*}\in X^{*}\},\ (x,x^{*})\in X\times X^{*}.\label{eq:*}
\end{equation}
Here ``$\min$'' stands for an infimum that is attained when finite. 

For every $z=(x,x^{*})\in[h^{\square}=c]\cap C\times X^{*}$ there
is $v^{*}\in X^{*}$ such that $\psi_{T|_{C}}(x,v^{*})+\sigma_{C}(x^{*}-v^{*})=\langle x,x^{*}\rangle$.
Since $\psi_{T|_{C}}\ge c$, $\sigma_{C}(x^{*}-v^{*})\ge\langle x,x^{*}-v^{*}\rangle$,
this implies that $(x,v^{*})\in[\psi_{T|_{C}}=c]\cap C\times X^{*}=T|_{C}$,
$x^{*}-v^{*}\in N_{C}(x)$, and $z\in T+N_{C}$. Therefore $[h^{\square}=c]\cap C\times X^{*}\subset T+N_{C}$.
The inclusion $[h=c]\subset[h^{\square}=c]$ completes the proof of
this implication (see Lemma \ref{h>c}). 

(iii) $\Leftrightarrow$ (iv) Since $(T+N_{C})|_{C}=T+N_{C}$, this
equivalence follows from (\ref{tnc}). \end{proof}

\strut

In the absence of the $C-$representability of $T$ the previous result
still holds with $T$ replaced by $R=[\psi_{T|_{C}}=c]\cap C\times X^{*}$
which is the smallest $C-$representable extension of $T$ in the
following string of implications: $C$ locates $T$ $\Rightarrow$
$T$ is $C-$NI $\Leftrightarrow$ $R$ is $C-$NI $\Leftrightarrow$
$C$ locates $R$ $\Leftrightarrow$ $C$ identifies $R+N_{C}$. The
converse of the first implication is false as seen from Remark \ref{Vbar->V?}
below for $C=\overline{V}$. 

\begin{corollary} \label{C-amt-cor}Let $X$ be a LCS, let $T:X\rightrightarrows X^{*}$,
and let $C\subset X$ be closed convex such that $D(T)\cap\operatorname*{int}C\neq\emptyset$
and $T$ is $C-$representable. Then $T+N_{C}\in\mathfrak{M}(X)$
iff $C$ locates $T$ and $[\varphi_{T+N_{C}}\le c]\subset C\times X^{*}$
iff $T$ is $C-$NI and $[\varphi_{T+N_{C}}\le c]\subset C\times X^{*}$.
\end{corollary}

\begin{proposition}\label{V->V bar} Let $X$ be a LCS, let $T:X\rightrightarrows X^{*}$,
and let $V\subset X$ be open convex such that $D(T)\cap V\neq\emptyset$
and $T|_{V}\in\mathcal{M}(X)$. If $T$ is $V-$NI then
\[
\varphi_{T|_{V}}\ge c,\ {\rm on}\ \overline{V}\times X^{*}.
\]
In particular, for every $V\subset S\subset\overline{V}$, $T$ is
$S-$NI.

If, in addition, $T$ is $\overline{V}-$representable, then 
\[
[\varphi_{T|_{V}}\le c]\cap\overline{V}\times X^{*}\subset T+N_{\overline{V}}\subset D(T)\times X^{*}.
\]

In particular, for every $V\subset S\subset\overline{V}$, $S$ locates
$T$ and identifies $T+N_{\overline{V}}$. \end{proposition} 

\begin{proof} Seeking a contradiction assume that there is $z=(x,x^{*})\in[\varphi_{T|_{V}}<c]\cap\overline{V}\times X^{*}$.
Since $T$ is $V-$NI we know that $x\in\operatorname*{bd}V$.

Let $y\in V\cap D(T)$, $w=(y,y^{*})\in T$, and $h:[0,1]\rightarrow\mathbb{R}$,
$h(t)=(\varphi_{T|_{V}}-c)(tw+(1-t)z)$. Note that $h(0)<0$, $h(1)=0$,
since $w\in T|_{V}\in\mathcal{M}(X)$, and $h$ is continuous. Hence
there is $\delta\in(0,1)$ such that $h(\delta)<0$. That provides
the contradiction $\delta w+(1-\delta)z\in[\varphi_{T|_{V}}<c]\cap V\times X^{*}$. 

Note that $h:=\varphi_{T|_{V}}+\iota_{\overline{V}\times X^{*}}\in\mathscr{R}$,
$[h=c]=[\varphi_{T|_{V}}\le c]\cap\overline{V}\times X^{*}$, and
\[
h^{\square}(x,x^{*})=\min\{\psi_{T|_{V}}(x,u^{*})+\sigma_{\overline{V}}(x^{*}-u^{*})\mid u^{*}\in X^{*}\},\ (x,x^{*})\in X\times X^{*}.
\]
For every $z=(x,x^{*})\in[h^{\square}=c]\cap\overline{V}\times X^{*}$
there is $v^{*}\in X^{*}$ such that $\psi_{T|_{V}}(x,v^{*})+\sigma_{\overline{V}}(x^{*}-v^{*})=\langle x,x^{*}\rangle$.
This implies that $(x,v^{*})\in[\psi_{T|_{V}}=c]\cap\overline{V}\times X^{*}\subset[\psi_{T|_{\overline{V}}}=c]\cap\overline{V}\times X^{*}$
and $x^{*}-v^{*}\in N_{\overline{V}}(x)$. If, in addition, $T$ is
$\overline{V}-$representable then $[\psi_{T|_{\overline{V}}}=c]\cap\overline{V}\times X^{*}=T|_{\overline{V}}$
so $z\in T+N_{\overline{V}}$. Hence 
\[
[\varphi_{T|_{V}}\le c]\cap\overline{V}\times X^{*}=[h=c]\subset[h^{\square}=c]\cap\overline{V}\times X^{*}\subset T+N_{\overline{V}}\subset D(T)\times X^{*}.
\]
\end{proof} 

\begin{proposition} \label{locV->locVbar}Let $X$ be a LCS, let
$T:X\rightrightarrows X^{*}$, and let $V\subset X$ be open convex
such that $D(T)\cap V\neq\emptyset$ and $T$ is $\overline{V}-$representable.
If $V$ locates $T$ then, for every $V\subset S\subset\overline{V}$,
$S$ locates $T$ and identifies $T+N_{\overline{V}}$. \end{proposition} 

\begin{remark} \label{Vbar->V?} \emph{Let $T=(0,1)\times\{0\}\subset\mathbb{R}^{2}$.
Then $V=(0,1)=D(T)$ is open convex and identifies $T$ (and, according
to Remark \ref{reprez}, that makes $T$ become $V-$representable)
while $\overline{V}=[0,1]$ does not locate $T$ since $z=(0,0)$
is m.r. to $T$ and $0\in\overline{V}\setminus D(T)$. }

\emph{This example shows the necessity of the $\overline{V}-$representability
condition in the previous two propositions and also, that this condition
cannot be replaced by $V-$representability. }\end{remark}

\begin{theorem} \label{open->closed}Let $X$ be a LCS and let $T:X\rightrightarrows X^{*}$. 

\medskip

\emph{(i)} $T$ is locally-NI iff, for every closed convex $C\subset X$
such that $D(T)\cap\operatorname*{int}C\neq\emptyset$, $T$ is $C-$NI.

\medskip

\emph{(ii)} $T$ is locatable iff, for every closed convex $C\subset X$
such that $D(T)\cap\operatorname*{int}C\neq\emptyset$, 
\[
[\varphi_{T|_{C}}\le c]\cap\operatorname*{int}C\times X^{*}\subset D(T)\times X^{*}.
\]
In particular if, for every closed convex $C\subset X$ such that
$D(T)\cap\operatorname*{int}C\neq\emptyset$, $C$ locates $T$ then
$T$ is locatable.

\medskip

\emph{(iii)} $T$ is identifiable iff, for every closed convex $C\subset X$
such that $D(T)\cap\operatorname*{int}C\neq\emptyset$,  
\[
[\varphi_{T|_{C}}\le c]\cap\operatorname*{int}C\times X^{*}\subset\operatorname*{Graph}(T),
\]
In particular if, for every closed convex $C\subset X$ such that
$D(T)\cap\operatorname*{int}C\neq\emptyset$, $C$ identifies $T+N_{C}$
then $T$ is identifiable. 

\medskip

\emph{(iv)} $T$ is monotone and identifiable iff $T$ is representable
and, for every closed convex $C\subset X$ with $D(T)\cap\operatorname*{int}C\neq\emptyset$,
$C$ locates $T$ iff $T$ is monotone low-representable and, for
every closed convex $C\subset X$ with $D(T)\cap\operatorname*{int}C\neq\emptyset$,
$T$ is $C-$NI.

\end{theorem}

\begin{proof} First we prove that for every open convex $V\subset X$
such that $D(T)\cap V\neq\emptyset$ and for every $x\in V$ there
is a closed convex $C\subset V$ such that $D(T)\cap\operatorname*{int}C\neq\emptyset$
and $x\in\operatorname*{int}C$. Indeed, take $y\in D(T)\cap V$ and
a closed convex $U\in\mathscr{{V}}(0)$ such that $C:=[x,y]+U\subset V$.
Note that $C$ is closed convex and $y\in D(T)\cap\operatorname*{int}C$.
The last inclusion is possible since if we assume the opposite, namely
that for every $U\in\mathscr{{V}}(0)$ there is $x_{U}\in([x,y]+U)\setminus V$,
that is, $x_{U}-y_{U}\in U$ for some $y_{U}\in[x,y]$, because $[x,y]$
is compact, on a subnet, denoted by the same index for notation simplicity,
$y_{U}\to y\in[x,y]\subset V$ and so we reach the contradiction $x_{U}\to y\not\in V$. 

(i) ($\Rightarrow$) For every closed convex $C\subset X$ such that
$D(T)\cap\operatorname*{int}C\neq\emptyset$, $T$ is $\operatorname*{int}C-$NI.
According to Proposition \ref{V->V bar}, $T$ is also $C=\operatorname*{cl}(\operatorname*{int}C)-$NI. 

($\Leftarrow$) For every $V\subset X$ open convex such that $D(T)\cap V\neq\emptyset$
and every $z=(x,x^{*})\in V\times X^{*}$ let $C\subset V$ be closed
convex such that $D(T)\cap\operatorname*{int}C\neq\emptyset$ and
$x\in C$. Since $T$ is $C-$NI and $T|_{C}\subset T|_{V}$ we get
$\varphi_{T|_{V}}(z)\ge\varphi_{T|_{C}}(z)\ge c(z)$, i.e., $T$ is
locally-NI. 

(ii) ($\Rightarrow$) For every closed convex $C\subset X$ such that
$D(T)\cap\operatorname*{int}C\neq\emptyset$, we have $[\varphi_{T|_{C}}\le c]\cap\operatorname*{int}C\times X^{*}\subset[\varphi_{T|_{\operatorname*{int}C}}\le c]\cap\operatorname*{int}C\times X^{*}\subset D(T)\times X^{*}$
since $\operatorname*{int}C$ locates $T$. 

($\Leftarrow$) For every $V\subset X$ open convex such that $D(T)\cap V\neq\emptyset$
and every $z=(x,x^{*})\in[\varphi_{T|_{V}}\le c]\cap V\times X^{*}$
let $C\subset V$ be closed convex such that $D(T)\cap\operatorname*{int}C\neq\emptyset$
and $x\in\operatorname*{int}C$. Then $z\in[\varphi_{T|_{C}}\le c]\cap\operatorname*{int}C\times X^{*}\subset D(T)\times X^{*}$.
This yields that $[\varphi_{T|_{V}}\le c]\cap V\times X^{*}\subset D(T)\times X^{*}$,
that is, $V$ locates $T$. 

The proof of (iii) is similar to the argument used for (ii). In particular
if, for every closed convex $C\subset X$ such that $D(T)\cap\operatorname*{int}C\neq\emptyset$,
$C$ identifies $T+N_{C}$ then for every $V\subset X$ open convex
such that $D(T)\cap V\neq\emptyset$ and every $z=(x,x^{*})\in[\varphi_{T|_{V}}\le c]\cap V\times X^{*}$
let $C\subset V$ be closed convex such that $D(T)\cap\operatorname*{int}C\neq\emptyset$
and $x\in\operatorname*{int}C$. Hence $z\in[\varphi_{T|_{C}}\le c]\cap\operatorname*{int}C\times X^{*}=[\varphi_{T+N_{C}}\le c]\cap\operatorname*{int}C\times X^{*}\subset\operatorname*{Graph}(T)$. 

Subpoint (iv) is a direct consequence of (i) and Theorems \ref{lo+re <-> id},
\ref{LowRep}, \ref{C-amt}. \end{proof} 

\eject

\section{Representability via the convolution operation}

The goal of this section is to study, in the context of locally convex
spaces, the representability of the sum $A+B$ of two representable
operators $A$, $B$ under classical qualification constraints. Under
a Banach space settings, the calculus rules of representable operators
can be found in \cite[Section\ 5]{MR2453098}.  The following result
holds

\begin{proposition} \label{zali1} (Z\u{a}linescu \cite[Proposition 1]{MR865564})
Let $X_{1}$, $X_{2}$ be LCS's and let $f_{1},f_{2}:X=X_{1}\times X_{2}\to\overline{\mathbb{R}}$
be proper convex functions. If there exists $(x_{1},x_{2})\in\operatorname*{dom}f_{1}\cap\operatorname*{dom}f_{2}$
such that $f_{1}(\cdot,x_{2})$ is continuous at $x_{1}$ and $f_{2}(x_{1},\cdot)$
is continuous at $x_{2}$ then, for every $x^{*}\in X^{*}=X_{1}^{*}\times X_{2}^{*}$
\[
(f_{1}+f_{2})^{*}(x^{*})=\min\{f_{1}^{*}(u^{*})+f_{2}^{*}(x^{*}-u^{*})\mid u^{*}\in X^{*}\}.
\]
Here ``$\min$'' stands for an infimum that is attained when finite.
\end{proposition}

\begin{theorem} \label{zali2} Let $E$, $F$ be LCS's and let $\phi_{1},\phi_{2}:E\times F\to\overline{\mathbb{R}}$
be proper convex functions. Consider $\rho:E\times F\to\overline{\mathbb{R}}$
defined by
\[
\rho(x,y):=\inf\{\phi_{1}(x,y_{1})+\phi_{2}(x,y_{2})\mid y_{1}+y_{2}=y\}.
\]

Assume that there exists $(x_{0},y_{0})\in\operatorname*{dom}\phi_{2}$
such that $x_{0}\in\operatorname*{Pr}{}_{E}(\operatorname*{dom}\phi_{1})$
and $\phi_{2}(\cdot,y_{0})$ is continuous at $x_{0}$. Then, for
every $x^{*}\in E^{*}$, $y^{*}\in F^{*}$

\[
\rho^{*}(x^{*},y^{*})=\min\{\phi_{1}^{*}(x_{1}^{*},y^{*})+\phi_{2}^{*}(x_{2}^{*},y^{*})\mid x_{1}^{*}+x_{2}^{*}=x^{*}\in X^{*}\}.
\]
\end{theorem}

\begin{proof} Note first that $\operatorname*{dom}\rho\neq\emptyset$
because there is $\tilde{y}\in F$ such that $(x_{0},\tilde{y})\in\operatorname*{dom}\phi_{1}$,
so $\rho(x_{0},\tilde{y}+y_{0})\le\phi_{1}(x_{0},\tilde{y})+\phi_{2}(x_{0},y_{0})<+\infty$;
whence $\rho^{*}$ does not take the value $-\infty$. 

For every $x\in E$, $y_{1},y_{2}\in F$, $x_{1}^{*},x_{2}^{*}\in E^{*}$,
$y^{*}\in F^{*}$ we have
\[
\phi_{1}(x,y_{1})+\phi_{2}(x,y_{2})+\phi_{1}^{*}(x_{1}^{*},y^{*})+\phi_{2}^{*}(x_{2}^{*},y^{*})\ge\langle x,x_{1}^{*}+x_{2}^{*}\rangle+\langle y_{1}+y_{2},y^{*}\rangle,
\]
so, for every $(x,y)\in E\times F$, $x_{1}^{*},x_{2}^{*}\in E^{*}$,
$y^{*}\in F^{*}$, 
\[
\rho(x,y)+\phi_{1}^{*}(x_{1}^{*},y^{*})+\phi_{2}^{*}(x_{2}^{*},y^{*})\ge\langle x,x_{1}^{*}+x_{2}^{*}\rangle+\langle y,y^{*}\rangle,
\]
from which, for every $x_{1}^{*},x_{2}^{*}\in E^{*}$, $y^{*}\in F^{*}$,
$\phi_{1}^{*}(x_{1}^{*},y^{*})+\phi_{2}^{*}(x_{2}^{*},y^{*})\ge\rho^{*}(x_{1}^{*}+x_{2}^{*},y^{*})$.
Hence
\[
\forall(x^{*},y^{*})\in E^{*}\times F^{*},\ \rho^{*}(x^{*},y^{*})\le\inf\{\phi_{1}^{*}(x_{1}^{*},y^{*})+\phi_{2}^{*}(x_{2}^{*},y^{*})\mid x_{1}^{*}+x_{2}^{*}=x^{*}\in X^{*}\}.
\]

If $\rho^{*}(x^{*},y^{*})=+\infty$ the conclusion holds. 

If $\rho^{*}(x^{*},y^{*})\in\mathbb{R}$ consider $f_{1},f_{2}:(E\times F)\times F\to\overline{\mathbb{R}}$
given by $f_{1}(x,y;z)=\phi_{2}(x,z)$, $f_{2}(x,y;z)=\phi_{1}(x,y)-\langle x,x^{*}\rangle-\langle y+z,y^{*}\rangle$.
Notice that, for every $u^{*}\in E^{*}$, $v^{*},z^{*}\in F^{*}$
\[
\begin{array}{ll}
f_{1}^{*}(u^{*},v^{*};z^{*})= & \phi_{2}(u^{*},z^{*})+\iota_{\{0\}}(v^{*})\\
f_{2}^{*}(u^{*},v^{*};z^{*})= & \phi_{1}(x^{*}+u^{*},y^{*}+v^{*})+\iota_{\{0\}}(y^{*}+z^{*}),
\end{array}
\]
\[
\begin{array}{ll}
(f_{1}+f_{2})^{*}(0) & =-\inf\{\phi_{1}(x,y)+\phi_{2}(x,z)-\langle x,x^{*}\rangle-\langle y+z,y^{*}\rangle\mid x\in E,y,z\in F\}\\
 & =-\inf\{\rho(x,v)-\langle x,x^{*}\rangle-\langle v,y^{*}\rangle\mid x\in E,v\in F\}=\rho^{*}(x^{*},y^{*})\in\mathbb{R}.
\end{array}
\]

Let $\zeta_{1}:=(x_{0},\tilde{y})\in\operatorname*{dom}\phi_{1}$
and let $\zeta_{2}=y_{0}$. Then $(\zeta_{1},\zeta_{2})\in\operatorname*{dom}f_{1}\cap\operatorname*{dom}f_{2}$,
$\zeta=(x,y)\to f_{1}(x,y;\zeta_{2})=\phi_{2}(x,y_{0})$ is continuous
at $\zeta_{1}$, and $f_{2}(x_{0},\tilde{y};\cdot)$ is continuous
at $\zeta_{2}$. From Proposition \ref{zali1} we obtain $u^{*}\in E^{*}$,
$v^{*},z^{*}\in F^{*}$ such that $\rho^{*}(x^{*},y^{*})=f_{1}^{*}(u^{*},v^{*};z^{*})+f_{2}^{*}(-u^{*},-v^{*};-z^{*})$,
i.e., $v^{*}=0$, $z^{*}=y^{*}$ and $\rho^{*}(x^{*},y^{*})=\phi_{1}(x^{*}-u^{*},y^{*})+\phi_{2}(u^{*},y^{*})$.
\end{proof}

\begin{theorem} \label{A+NC repres} Let $X$ be a LCS, let $A:X\rightrightarrows X^{*}$
be representable, and let $C\subset X$ be closed convex. If $D(A)\cap\operatorname*{int}C\neq\emptyset$
then $A+N_{C}$ is representable. \end{theorem} 

\begin{proof} Let $x_{0}\in D(A)\cap\operatorname*{int}C$, $a_{0}^{*}\in Ax_{0}$.
We apply the previous theorem for $E=X$, $F=(X^{*},w^{*})$, $\phi_{1}=\varphi_{A}$,
$\phi_{2}(x,x^{*})=\iota_{C}(x)+\sigma_{C}(x^{*})$, $(x,x^{*})\in X\times X^{*}$,
and $y_{0}=0\in\operatorname*{dom}\sigma_{C}$ to get that 
\[
\rho(x,x^{*}):=\inf\{\varphi_{A}(x,x_{1}^{*})+\iota_{C}(x)+\sigma_{C}(x_{2}^{*})\mid x_{1}^{*}+x_{2}^{*}=x^{*}\},\ (x,x^{*})\in X\times X^{*},
\]
has
\[
\rho^{\square}(x,x^{*}):=\min\{\psi_{A}(x,x_{1}^{*})+\iota_{C}(x)+\sigma_{C}(x_{2}^{*})\mid x_{1}^{*}+x_{2}^{*}=x^{*}\},\ (x,x^{*})\in X\times X^{*}.
\]
which is a representative of $A+N_{C}$. \end{proof}

\begin{theorem} \label{A+B repres} Let $X$ be a barreled LCS and
let $A,B:X\rightrightarrows X^{*}$ be representable such that $D(A)\cap\operatorname*{int}D(B)\neq\emptyset$.
For every $x\in X$, there exist $V\in\mathscr{{V}}(x)$, such that
$A+B$ is $V-$representable. In particular, $A+B$ is low-representable.
\end{theorem} 

\begin{proof} Fix 
\[
x_{0}\in D(A)\cap\operatorname*{int}D(B),\ a_{0}^{*}\in Ax_{0},\ b_{0}^{*}\in Bx_{0},\ x_{0}^{*}=a_{0}^{*}+b_{0}^{*},\ z_{0}:=(x_{0},x_{0}^{*})\in A+B,
\]
and symmetric open convex $U_{0},U_{1}\in\mathscr{{V}}(0)$ such that
$U_{1}+U_{1}\subset U_{0}$, $x_{0}+U_{0}\subset D(B)$, and $B(x_{0}+U_{0})$
is equicontinuous; for simplicity $B(x_{0}+U_{0})\subset U_{0}^{\circ}$. 

Take an arbitrary $x\in X$ and denote by $V=[x_{0},x]+U_{1}\in\mathscr{{V}}(x)$.

Notice that $(x_{0},a_{0}^{*})\in A|_{V}$ so $x_{0}\in\operatorname*{Pr}{}_{X}(\operatorname*{dom}\varphi_{A|_{V}})$
and $\varphi_{B|_{V}}(\cdot,0)$ is continuous at $x_{0}$ since $\varphi_{B|_{V}}(\cdot,0)$
is bounded from above on $x_{0}+U_{1}$. Indeed, for every $y\in x_{0}+U_{1}\subset D(B)$,
$y^{*}\in B(y)$, $b\in V\cap D(B)$, $b^{*}\in B(b)$, we have $y-b\in[0,x_{0}-x]+U_{0}$
and $\langle y-b,b^{*}\rangle\le\langle y-b,y^{*}\rangle$. Hence,
for every $y\in x_{0}+U_{1}$, $y^{*}\in B(y)$
\[
\begin{array}{ll}
\varphi_{B|_{V}}(y,0) & =\sup\{\langle y-b,b^{*}\rangle\mid b\in V\cap D(B),\ b^{*}\in B(b)\}\\
 & \le\sup\{\langle y-b,y^{*}\rangle\mid b\in V\cap D(B),\ b^{*}\in B(b)\}\\
 & \le\sup\{|\langle x_{0}-x,u^{*}\rangle|+1\mid u^{*}\in B(x_{0}+U_{0})\}<+\infty.
\end{array}
\]
 Consider $\rho:X\times X^{*}\to\overline{\mathbb{R}}$, 
\[
\rho(y,y^{*})=\inf\{\varphi_{A|_{V}}(y,u^{*})+\varphi_{B|_{V}}(y,v^{*})\mid u^{*}+v^{*}=y^{*}\}.
\]
We apply Proposition \ref{zali2} for $E=X$, $F=(X^{*},w^{*})$ to
get
\[
\rho^{\square}(y,y^{*})=\min\{\psi_{A|_{V}}(y,u^{*})+\psi_{B|_{V}}(y,v^{*})\mid u^{*}+v^{*}=y^{*}\}.
\]

Note that $\varphi_{A+B|_{V}}\le\rho$, so, $\rho^{\square}\le\psi_{A+B|_{V}}$.
Therefore, for every $w=(y,y^{*})\in[\psi_{A+B|_{V}}=c]\cap V\times X^{*}$
there exists $u^{*}\in X^{*}$ such that $(y,u^{*})\in[\psi_{A|_{V}}=c]$,
$(y,x^{*}-u^{*})\in[\psi_{B|_{V}}=c]$, since $\rho^{\square}\ge c$.
Because $y\in V$ and $A,B$ are $V-$representable, we get that $w\in A+B$,
that is, $A+B$ is $V-$representable. \end{proof}

\section{Open problems}
\begin{enumerate}
\item If $T$ is maximal monotone, how can we better describe condition
(\ref{C})? Is it the same as $T$ is locally-NI?
\item For a fixed open convex $V$, clearly $V$ locates $T$ implies that
$T$ is $V-$NI which in turns yields condition (\ref{C}). Are any
of the converses of these two implications true?
\item If $T$ is monotone then (\ref{C}) $\Rightarrow$ $\overline{D(T)}$
is convex?
\item If $T$ is monotone, $V$ is open convex, $V\cap D(T)\neq\emptyset$,
and $\overline{V}$ locates $T$ then must $V$ locate $T$?
\item Under the hypotheses of Theorem \ref{C-amt}, is $T+N_{C}$ maximal
monotone?
\end{enumerate}

\eject

\bibliographystyle{plain}

\end{document}